\documentclass[11pt]{amsart}
\usepackage{amsmath,latexsym,amssymb,amsfonts,color,hyperref,graphics,amsthm,enumerate}
\usepackage[margin=3cm]{geometry}

\newcommand{\Aut}{\mathrm{Aut}}

\newcommand{\Out}{\mathrm{Out}}

\newcommand{\Soc}{\mathrm{Soc}}
\newcommand{\AGL}{\mathrm{AGL}}
\newcommand{\Sp}{\mathrm{Sp}}

\newcommand{\Fun}{\mathrm{Fun}}

\newcommand{\PSL}{\mathrm{PSL}}

\newcommand{\PSU}{\mathrm{PSU}}
\newcommand{\PGL}{\mathrm{PGL}}

\newcommand{\N}{\mathrm{N}}
\newcommand{\Sym}{\mathrm{Sym}}
\newcommand{\A}{\mathrm{A}}
\newcommand{\Z}{\mathrm{Z}}
\newcommand{\D}{\mathrm{D}}

\newcommand{\CF}{\mathrm{CF}}
\newcommand{\C}{\mathrm{C}}

\newcommand{\ICF}{\mathrm{InsolCF}}
\newcommand{\m}{\mathrm{Mult}}

\newtheorem{thm}{Theorem}[section]
\newtheorem{lem}[thm]{Lemma}

\theoremstyle{definition}

\newtheorem{que}[thm]{Question}
\newtheorem{rem}[thm]{Remark}
\newtheorem{hyp}[thm]{Hypothesis}

\begin{document}

\title[Vertex-primitive $s$-arc-transitive digraphs]{Bounding $s$ for vertex-primitive $s$-arc-transitive digraphs of alternating and symmetric groups}
\thanks{This work is supported by NNSFC (11861012, 12061092), the special foundation for Guangxi Ba Gui Scholars and Yunnan Applied Basic Research Projects (202101AT070137).}
\thanks{$^*$Corresponding author. Email address: lijjhx@gxu.edu.cn}

\author{Junyan Chen}
\address[Junyan Chen]{Department of Mathematics, Southern University of Science and Technology, Shenzhen, 518055, P.R.~China.}

\author{Lei Chen}
\address[Lei Chen, Michael Giudici, Cheryl E. Praeger]{Department of Mathematics and Statistics, The University of Western Australia, Perth WA 6009}

\author{Michael Giudici}

\author{Jing Jian Li$^*$}
\address[Jing Jian Li]{College of Mathematics and Information Science, Guangxi University, Nanning, 530004, P.R.~China.}
%\email{chenjunyan2021@163.com (Chen); lijjhx@gxu.edu.cn (Li)}

\author{Cheryl E. Praeger}

\author{Binzhou Xia}
\address[Binzhou Xia]{School of Mathematics and Statistics, University of Melbourne, Parkville, VIC 3010, Australia.}
%\email{binzhoux@unimelb.edu.au}

\begin{abstract}
Determining an upper bound on $s$ for finite vertex-primitive $s$-arc-transitive digraphs has received considerable attention dating back to a question of Praeger in 1990. It was shown by Giudici and Xia that the smallest upper bound on $s$ is attained for some digraph admitting an almost simple \(s\)-arc-transitive group. In this paper, based on the work of Pan, Wu and Yin, we prove that $s\leqslant 2$ in the case where the group is an  alternating or symmetric group.

\end{abstract}

\maketitle

\section{Introduction}

A digraph $\Gamma$ is a pair $(V,\rightarrow)$ with $V$ a set and~$\rightarrow$  an antisymmetric irreflexive binary relationon $V$.
Let $s$ be a positive integer.
An \emph{$s$-arc} is a sequence $v_0,v_1,\cdots,v_s$ of vertices such that $v_i\rightarrow v_{i+1}$ for $0\leqslant i\leqslant s-1$.
Let $G$ be a subgroup of the automorphism group $\Aut(\Gamma)$ of~$\Gamma$.
The digraph $\Gamma$ is said to be \emph{$G$-vertex-primitive} if $G$ acts primitively on $V$, and \emph{$(G,s)$-arc-transitive} if $G$ acts transitively on the set of $s$-arcs.
Note that a $(G,s+1)$-arc-transitive digraph is necessarily $(G,s)$-arc-transitive.
We call a digraph $\Gamma$ \emph{vertex-primitive} (or \emph{$s$-arc-transitive}) if $\Gamma$ is $\Aut(\Gamma)$-vertex-primitive (or $(\Aut(\Gamma),s)$-arc-transitive).

In 1990, Praeger~\cite{Praeger1990} asked whether there exists a finite vertex-primitive $2$-arc-transitive digraph other than directed cycles.
This was answered in the affirmative in 2017 by Giudici, Li and Xia~\cite{GLX2017}, with a construction of an infinite family of such digraphs and these authors asked the following question.

\begin{que}\label{que1}
Is there an upper bound on $s$ for finite vertex-primitive $s$-arc-transitive digraphs that are not directed cycles?
\end{que}

A group $G$ is said to be \emph{almost simple} if~$T\leqslant G\leqslant \Aut(T)$ for some nonabelian simple group $T$.
By a result of Giudici and Xia~\cite[Corollary 1.6]{GX2018}, Question~\ref{que1} is reduced to the case where $\Aut(\Gamma)$ is almost simple. Following this, it is shown in~\cite{GLX2019} that $s\leqslant 2$ for $G$-vertex-primitive $(G,s)$-arc-transitive digraphs where $G$ is almost simple with socle a projective linear group.
More recently, Pan, Wu and Yin~\cite{PWY2020} studied the case where the socle of $G$ is an alternating group, proving the following result.

\begin{thm}[Pan-Wu-Yin]\label{PWY}
Let $\Gamma$ be a $G$-vertex-primitive $(G,s)$-arc-transitive digraph where $G$ is almost simple with socle $A_n$, and let $v$ be a vertex of $\Gamma$. Then one of the following holds:
\begin{enumerate}[{\rm (1)}]
\item \(s\leqslant2\);
\item $(A_m\wr S_k)\cap G\leqslant G_v\leqslant (S_m\wr S_k)\cap G$ with $m\geqslant 8$, $k>1$ and $n=mk$, $m^k$ or $(m!/2)^{k-1}$.
\end{enumerate}
\end{thm}

In this paper, we improve the above result of Pan, Wu and Yin by showing that for Case~(2) of Theorem~\ref{PWY} we still have $s\leqslant 2$. This establishes the bound $s\leqslant 2$ for all such $G$-vertex-primitive $(G,s)$-arc-transitive digraphs. Moreover, we give necessary conditions for $s=2$ to be attained. Note that, a recent result \cite{smallest} shows that the smallest \(G\)-vertex-primitive \((G,2)\)-arc-transitive digraph has 30758154560 vertices. Since \(30758154560>13!=|S_{13}|\), we have \(n\geqslant 14\) for any $G$-vertex-primitive $(G,s)$-arc-transitive digraph with $s\geqslant2$ such that $G$ is almost simple with socle $A_n$. In particular, such a group $G$ is \(A_{n}\) or \(S_{n}\). Our main result is as follows.

\begin{thm}\label{thmm}
Let $\Gamma$ be a $G$-vertex-primitive $(G,s)$-arc-transitive digraph such that $s\geqslant2$ and $G$ is almost simple with socle $A_n$, and let \(v\) be a vertex of $\Gamma$. Then \(s=2\), $G_v$ is primitive in its natural action on $n$ points, and one of the following holds:
\begin{enumerate}[{\rm (a)}]
     \item $T\trianglelefteq G_v\leqslant \Aut(T)$ with $T=\Sp_4(2^f)$ or $\mathrm{P\Omega}_8^+(q)$, where $f\geqslant2$ and $q$ is a prime power;
     \item $G_{v}=(T^k.(\Out(T)\times S_k))\cap G$ with $n=|T|^{k-1}$ for some nonabelian simple group $T$ and integer $k\geqslant 2$ such that $k=2$ if $T$ is an alternating group.
\end{enumerate}
\end{thm}

\begin{rem}
For almost simple groups $G$ with socle $A_n$, we do not know of any examples of a $G$-vertex-primitive $(G,2)$-arc-transitive digraph.
However, the existence of $G$-arc-transitive digraphs with $s=1$ can be seen in the following way. Let $H$ be a core-free maximal subgroup of $G$, and let $\chi$ be the permutation character of the action of $G$ on $[G:H]$ by right multiplication, where $[G:H]$ denotes the set of right cosets of $H$ in $G$. Every orbital digraph of $G$ on $[G:H]$ is self-paired if and only if $\chi$ is multiplicity-free and real (see~\cite[Page~45]{Cameron1999}). Thus, if $\chi$ is not multiplicity-free, then there exists a non-self-paired orbital digraph of $G$ on $[G:H]$, which is then a $G$-vertex-primitive $(G,1)$-arc-transitive digraph. For example, a result of Saxl~\cite{saxl} (see also~\cite[Theorem~1.1]{GM2010}) shows that, whenever $G=S_n$ with $n>18$ and $H$ is primitive in its natural action on $n$ points, then $\chi$ is not multiplicity-free.
\end{rem}

Theorem~\ref{thmm} will be proved in Section~\ref{3}, following the groundwork laid in Sections~\ref{2}--\ref{4}. Notably, Section~\ref{5} is dedicated to addressing a gap in the argument of~\cite{PWY2020} to prove Theorem~\ref{PWY} (see Remark~\ref{rem}).

\section{Notation and Preliminaries}\label{2}

For a group \(G\), denote the set of prime divisors of $|G|$ by $\Pi(G)$, the set of composition factors of $G$ by \(\CF(G)\), and the set of insoluble composition factors of $G$ by \(\ICF(G)\). For a simple group \(T\), let \(\m_{G}(T)\) denote the multiplicity of \(T\) as a composition factor of \(G\).
For a positive integer $n$ and prime number $p$, let $n_p$ denote the $p$-part of $n$ (the largest $p$-power dividing $n$).
The following elementary result is a consequence of Legendre’s formula.

\begin{lem}\label{p}
For any positive integer $n$ and prime $p$ we have $(n!)_p<p^{\frac{n}{p-1}}$.
\end{lem}

The next result is the well-known Bertrand's Postulate, see~\cite{hardy}, for example.

\begin{lem}\label{lem:bert}
For every real number \(x\geqslant 7\), there exists a prime number \(p\) satisfying \(x/2<p\leqslant x-2\).
\end{lem}

It is well known that the number of involutions in a nonabelian simple group is odd, see for instance~\cite[Theorem]{marcel2}. This together with~\cite[Theorem 1]{marcel} gives the result below.

\begin{lem}\label{marcel}
\emph{(Herzog)} Let \(T\) be a finite nonabelian simple group with \(I\) involutions. Then either \(I\equiv 3\pmod{4}\), or $T$ is one of the following groups and $I\equiv1\pmod{4}$.
\begin{enumerate}[{\rm(a)}]
    \item \(T=\PSL_{2}(q)\) with prime power \(q\equiv\varepsilon\pmod{8}\) for some \(\varepsilon=\pm1\), and \(I=q(q+\varepsilon)/2\);
    \item \(T=\PSL_{3}(q)\) with prime power \(q\equiv-1\pmod{4}\), and \(I=q^{2}(q^{2}+q+1)\);
    \item \(T=\PSU_{3}(q)\) with prime power \(q\equiv1\pmod{4}\), and \(I=q^{2}(q^{2}-q+1)\);
    \item \(T=A_{7}\) and \(I=105\);
    \item \(T=M_{11}\) and \(I=165\).
\end{enumerate}
\end{lem}

Let us recall the O'Nan-Scott Theorem (see, for example,~\cite[Theorem 7.11]{csaba}) on maximal subgroups of alternating and symmetric groups.

\begin{thm}[O'Nan-Scott]\label{onanscottthm}
Let $G=A_n$ or $S_n$, and let $H$ be a maximal subgroup of $G$ with $H\neq A_n$. Then one of the following holds:
\begin{enumerate}[{\rm(a)}]
\item $H=(S_k\times S_{n-k})\cap G$ with $1\leqslant k<n/2$;
\item $H=(S_m\wr S_k)\cap G$ with $n=mk$ for some integers $m\geqslant2$ and $k\geqslant2$;
\item $H=(S_m\wr S_k)\cap G$ with $n=m^k$ for some integers $m\geqslant 5$ and $k\geqslant2$;
\item $H=\AGL(k,p)\cap G$ with $n=p^k$ for some prime $p$ and positive integer $k$;
\item $H=(T^k.(\Out(T)\times S_k))\cap G$ with $n=|T|^{k-1}$ for some nonabelian simple group $T$ and integer $k\geqslant 2$;
\item $T\trianglelefteq H\leqslant \Aut(T)$ for some nonabelian simple group $T$.
\end{enumerate}
\end{thm}

An expression $G=AB$ of a group $G$ as the product of two subgroups $A$ and $B$ is called a \emph{factorization} of $G$.
An observation on group factorizations is as follows.

\begin{lem}\label{hab1}
If $G=AB$ is a factorization, then $G=A^xB^y$ is a factorization for all $x,y\in G$.
\end{lem}

The factorizations of groups between an alternating group and its automorphism group are classified by Liebeck, Praeger and Saxl~\cite[Theorem D]{LPS} (partial results can be found in the earlier work of Wiegold and Williamson~\cite{WW}), as stated below. A permutation group on a set \(\Omega\) is said to be \emph{\(k\)-homogeneous} if it is transitive on the set of \(k\)-subsets of \(\Omega\).

\begin{thm}[Liebeck-Praeger-Saxl]\label{LPS}
Let \(L=\A_{n}\) with \(n\geqslant5\), and let \(L\trianglelefteq G\leqslant\Aut(L)\). Suppose that \(G=AB\) with subgroups \(A\) and \(B\) of \(G\) not containing \(L\). Then, interchanging $A$ and $B$ if necessary, one of the following holds:
 \begin{enumerate}[{\rm (a)}]
 \item \(A_{n-k}\trianglelefteq A\leqslant S_{n-k}\times S_{k}\) for some \(k\) with \(1\leqslant k\leqslant5\), and \(B\) is \(k\)-homogeneous;
 \item \(n=10\), \(A=\PSL_{2}(8)\) or \(\PSL_{2}(8).3\), and \(A_{5}\times A_{5}\trianglelefteq B\leqslant S_{5}\wr S_{2}\) with \(B\) transitive on \(\Omega\);
 \item \(n=8\), \(\Z_{5}\times\Z_{3}\leqslant A\) and \(B=\AGL_{3}(2)\);
 \item \(n=6\), and one of the following holds:
     \begin{itemize}
         \item  \(A\cap L=\PSL_{2}(5)\), \(B\cap L\leqslant S_{3}\wr S_{2}\) with \(A\cap S_{6}\) and \(B\cap S_{6}\) both transitive on \(\Omega\);
         \item \(A\cap L=\Z_{5}\) or \(\D_{10}\) and \(B\cap L\leqslant S_{3}\wr S_{2}\);
         \item  \(G\) is not contained in \(S_{6}\), while \(A\cap S_{6}\) and \(B\cap S_{6}\) are as in \emph{(a)}.
     \end{itemize}
 \end{enumerate}
 \end{thm}

An immediate corollary of Theorem~\ref{LPS} is as follows.

\begin{lem}\label{ab}\cite[Lemma 2.3]{GLX2019}
Let \(G=S_{n}\) with \(n\geqslant2\) or $G=A_n$ with \(n\geqslant3\). Suppose \(G=AB\) with subgroups \(A\) and \(B\) of \(G\). Then at least one of \(A\) or \(B\) is transitive on \(\Omega\).
\end{lem}

\begin{rem}\label{rem}
In \cite{GLX2019} the statement of the above lemma erroneously included $G=A_2$, and this erroneous statement was unfortunately applied in~\cite{PWY2020} to prove Theorem~\ref{PWY}, when dealing with vertex stabilisers as maximal subgroups $H$ of $A_n$ in Cases~(b),~(c) and~(e) of Theorem~\ref{onanscottthm}. A gap in the reasoning would occur in these cases if $k=2$ and the projection of $H$ to $S_k$ is $A_k$ instead of $S_k$. In the next section we will show that the projection is $S_k$ in most cases, and the remaining cases will be handled within the first paragraph in the proof of Theorem~\ref{thmm} in Section~\ref{3}, thus fixing the gap in the proof of \cite[Theorem 1.1]{PWY2020}.
\end{rem}

We will need the following technical lemma in our analysis.

\begin{lem}\label{conju}
Let $G$ be an almost simple group with socle $A_m$ for some $m\geqslant 8$, let $q$ be the largest prime less than $m$, and let $\ell$ be an integer such that $q\leqslant\ell<m$. Then all subgroups of $G$ isomorphic to $A_\ell$ are conjugate in $G$.
\end{lem}

\begin{proof}
Regard $G$ as a permutation group on a set $\Omega$ with size $m$. Let $H$ be a subgroup of $G$ isomorphic to $A_\ell$.

Suppose that $H$ has an orbit $\Delta$ with $|\Delta|=d>\ell$. Let $X=\mathrm{Sym}(\Delta)$.
Then $H$ is a transitive subgroup of $X$, and so we have a factorisation $X=HY$, where $Y=S_{d-1}$ is a point-stabiliser of $X$ on $\Delta$.
Since $q\leqslant\ell<d$ and $q$ is the largest prime less than $d$, it holds that $\Pi(H)=\Pi(A_\ell)=\Pi(S_d)=\Pi(X)$.
Moreover, as $d-1\geqslant q$, we have $\Pi(Y)=\Pi(S_{d-1})=\Pi(S_d)=\Pi(X)$.
Note that $d>q\geqslant 7$ as $m\geqslant 8$. We conclude from~\cite[Theorem 1.1]{BP1998} that $X=S_d$ cannot have a factorisation $X=HY$ with $\Pi(H)=\Pi(Y)=\Pi(X)$.

Thus every orbit of $H$ has size at most $\ell$. This together with $H\cong A_\ell$ implies that the orbit sizes of $H$ are $\ell$ or $1$.
Since $q$ is the largest prime less than $m$, we have $q>m/2$ by Lemma~\ref{lem:bert}.
As a consequence, $\ell>m/2$, and so $H$ has exactly one orbit of size $\ell$ and $m-\ell$ orbits of size $1$.
In other words, $H$ has an orbit $\Delta$ of size $\ell$ and pointwise fixes $\Omega\setminus\Delta$. Thus $H\leqslant G_{(\Omega\setminus\Delta)}$, the pointwise stabiliser of $\Omega\setminus\Delta$, and \(G_{(\Omega\setminus\Delta)}\cong A_{\ell}\), so \(H=G_{(\Omega\setminus\Delta)}\). The conclusion follows since \(G\) is transitive on the set of \(\ell\)-subsets of \(\Omega\).
\end{proof}

\section{On maximal subgroups of alternating groups}\label{5}

For an imprimitive maximal subgroup \(S_m\wr S_k\) of \(S_{n}\), where $n=mk$ with \(m,k\geqslant2\). We view \(S_m\wr S_k\) as the permutation group \(\Sym(\Delta)\wr \Sym(\Gamma)\) on the set $\Delta\times\Gamma$, where $\Delta=\{1,\dots,m\}$ and $\Gamma=\{1,\dots,k\}$. More precisely, for \(f\in \Fun(\Gamma, \Sym(\Delta))\) (the set of functions from $\Gamma$ to $\Sym(\Delta)$), \(u\in\Sym(\Gamma)\) and \((\delta,\gamma)\in\Delta\times \Gamma\), define
\[
(\delta,\gamma)^{(f,u)}=(\delta^{f(\gamma)},\gamma^{u}).
\]

\begin{lem}\label{imprimitive}
Suppose that \(H=(S_{m}\wr S_{k})\cap A_n\) is an imprimitive maximal subgroup of \(A_{n}\), where \(n=mk\) and \(m,k\geqslant2\). Let \(M=S_{m}^{k}\) and \(\pi\) be the projection from \(H\) to \(HM/M\). Then \(\pi(H)=S_{k}\).
\end{lem}

\begin{proof}
Let $W=\Fun(\Gamma,\Sym(\Delta))\rtimes\Sym(\Gamma)=S_{m}\wr S_{k}$, in the above defined action on $\Delta\times\Gamma$. Since $\pi(W)=S_k$ and $|W:H|\leqslant 2$, we have $\pi(H)\geqslant A_k$. Thus it suffices to find a pair \((f,u)\in H\) such that \(u=(1,2)\in S_{k}\).

First assume that \(m\) is even. Take \((f,u)=((1,1,\ldots,1),(1,2))\in W\). Then \((f,u)\) swaps the~\(m\) pairs \((\delta,1)\) and \((\delta,2)\) for \(\delta\in\Delta\) and fixes all other pairs in \(\Delta\times\Gamma\). Thus \((f,u)\) is an even permutation of \(\Delta\times \Gamma\) and so \((f,u)\in H\).

Next assume that \(m\) is odd. Take \((f,u)=(((1,2),1,\ldots,1),(1,2))\). Then \((f,u)\) fixes the pairs \((\delta,\gamma)\) for \(\gamma\in\Gamma\setminus\{1,2\}\), swaps the pairs \((\delta,1)\) and \((\delta,2)\) for \(\delta\in\Delta\setminus\{1,2\}\), and  the orbit of \((1,1)\) under the action of \(\langle(f,u)\rangle\) is \(\{(1,1),(2,2),(2,1),(1,2)\}\) is of size \(4\). Again \((f,u)\) is an even permutation of \(\Delta\times\Gamma\), so \((f,u)\in H\).
\end{proof}

As shown in the next lemma, the conclusion of Lemma~\ref{imprimitive} holds for the primitive wreath product $S_{m}\wr S_{k}$.
Such a maximal subgroup \(S_{m}\wr S_{k}\) of \(S_{n}\), where $n=m^k$, can be viewed as the permutation group $\Sym(\Delta)\wr \Sym(\Gamma)$ on $\Fun(\Gamma, \Delta)$, where $\Delta=\{1,\dots,m\}$ and $\Gamma=\{1,\dots,k\}$. More precisely, for \(f\in \Fun(\Gamma, \Delta)\) and \(x=((x_{1},\ldots, x_{k}),\tau^{-1})\in \Sym(\Delta)\wr \Sym(\Gamma)\), where $(x_{1},\ldots, x_{k})\in\Fun(\Gamma,\Sym(\Delta))$ and $\tau\in\Sym(\Gamma)$, define
\[
f^{x}=(f(\tau(1))^{x_{\tau(1)}},\ldots, f(\tau(k))^{x_{\tau(k)}}).
\]

\begin{lem}\label{PA}
Suppose that \(H=(S_{m}\wr S_{k})\cap A_n\) is a maximal subgroup of product action type in \(A_{n}\), where \(n=m^{k}\) and \(m,k\geqslant2\). Let \(M=S_{m}^{k}\) and \(\pi\) be the projection from \(H\) to \(HM/M\). Then \(\pi(H)=S_{k}\).
\end{lem}

\begin{proof}
Let $W=\Fun(\Gamma,\Sym(\Delta))\rtimes\Sym(\Gamma)=S_{m}\wr S_{k}$, in the above defined action on $\Fun(\Gamma, \Delta)$.

If $k=2$ and $\pi(H)\neq S_2$ then $H$ preserves two different partitions of \(\Fun(\Gamma,\Delta)=\Delta^{2}\) into $m$ parts of size $m$, namely, \(\Pi_{1}=\{\Delta\times\{i\} \mid 1\leqslant i\leqslant m\}\) and \(\Pi_{2}=\{\{j\}\times \Delta \mid 1\leqslant j\leqslant m\}\), and so \(H\) is a proper subgroup of $(S_m\wr S_2)\cap\A_n$ by Lemma~\ref{imprimitive}, contradicting the maximality of $H$ in $A_n$. Thus from now on we assume that $k\geqslant 3$. Since $\pi(W)=S_k$ and $|W:H|\leqslant 2$, we have $\pi(H)\geqslant A_k$. Thus it suffices to show that $(1,2)\in\pi(H)$.

First assume that  \(m\equiv 0,1\) or \(2\pmod{4}\). Let us take \(x=(1,\ldots,1)(1,2)\). Then \(x\) swaps the pairs \(\{f,g\}\) such that \(f(1)=g(2)\neq f(2)=g(1)\) and \(f(i)=g(i)\) for \(i\in\{3,\ldots,k\}\) and fixes all other elements of \(\Fun(\Gamma,\Delta)\). There are \(m^{k-1}(m-1)/2\) such pairs, and since \(m^{k-1}(m-1)/2\) is even, \(x\) is an even permutation, and so \(x\in H\).

Next assume that \(m\equiv3\pmod{4}\). Take \(x=(((1,2),1,\ldots,1),(1,2))\). Then
\[
f^{x}=(f(2),f(1)^{(1,2)},f(3),\ldots, f(k)).
\]
We observe the following three sets of orbits of $\langle x\rangle$, denoted by \(\mathcal{O}_{1}\), \(\mathcal{O}_{2}\) and \(\mathcal{O}_{3}\).

The first set \(\mathcal{O}_{1}\) denotes the set of orbits of the form \(f^{\langle x\rangle}\) with \(f(1)=f(2)\geqslant3\). Here
\[
f^{x}=(f(2),f(1),\ldots, f(k))=(f(1), f(2),\ldots, f(k))=f.
\]
Hence each orbit in $\mathcal{O}_{1}$ is a singleton.
Since \(f(1), f(2)\in\{3,\ldots,m\}\) and \(f(i)\in\{1,\ldots,m\}\), it follows that \(|\mathcal{O}_1|=(m-2)m^{k-2}\).

The second set \(\mathcal{O}_{2}\) denotes the set of orbits of the form \(f^{\langle x\rangle}\) with \(f(1)\geqslant3\), \(f(2)\geqslant3\) and \(f(1)\neq f(2)\). For such an $f$, we have
\begin{align*}
     f^{x}&=(f(2),f(1),f(3),\ldots,f(k))\neq f, \\
     f^{x^{2}}&=(f(1),f(2),f(3),\ldots,f(k))=f.
\end{align*}
Hence each orbit in $\mathcal{O}_{2}$ has size $2$, and is uniquely determined by $\{f(1),f(2)\}$ and $(f(3),\dots,f(k))$.
Since \(f(1), f(2)\in\{3,\ldots,m\}\) with \(f(1)\neq f(2)\) and \(f(3),\dots,f(k)\in\{1,\ldots,m\}\), we deduce that $|\mathcal{O}_2|=(m-2)(m-3)m^{k-2}/2$.

The third set \(\mathcal{O}_{3}\) denotes the set of orbits of the form \(f^{\langle x\rangle}\) with \(f(1)=1\) and \(f(2)\geqslant2\). Here
\begin{align*}
         f^{x}&=(f(2), 2, f(3),\ldots,f(k)),\\
         f^{x^{2}}&=(2,f(2)^{(1,2)},f(3),\ldots,f(k))\neq f,\\
         f^{x^{3}}&=(f(2)^{(1,2)},1,f(3),\ldots,f(k)),\\
        f^{x^{4}}&=(1,f(2),f(3),\ldots,f(k))=f.
\end{align*}
Hence each orbit in $\mathcal{O}_{3}$ has size $4$, and is uniquely determined by $(f(2),f(3),\dots,f(k))$.
Since $f(2)\in\{2,\dots,m\}$ and \(f(3),\dots,f(k)\in\{1,\ldots,m\}\), it follows that \(|\mathcal{O}_3|=(m-1)m^{k-2}\).

Since orbits in $\mathcal{O}_{1}$, $\mathcal{O}_{2}$ and $\mathcal{O}_{3}$ have size $1$, $2$ and $4$, respectively, and
\[
|\mathcal{O}_{1}|+2|\mathcal{O}_{2}|+4|\mathcal{O}_{3}|=(m-2)m^{k-2}+2\cdot\frac{(m-2)(m-3)m^{k-2}}{2}+4(m-1)m^{k-2}=m^k=|\Fun(\Gamma,\Delta)|,
\]
we conclude that $\mathcal{O}_{1}$, $\mathcal{O}_{2}$ and $\mathcal{O}_{3}$ form a partition of the orbits of $\langle x\rangle$. Thus \(x\) has
\[
\frac{(m-2)(m-3)m^{k-2}}{2}+(m-1)m^{k-2}=\left(\frac{m(m-3)}{2}+2\right)m^{k-2}
\]
cycles of even length. Since \(m\equiv 3\pmod{4}\), this is an even number, so \(x\) is an even permutation and \(x\in H\).
\end{proof}

Recall that the maximal subgroups in Theorem~\ref{onanscottthm}\,(e) are of the so-called \emph{diagonal type}.

\begin{lem}\label{PSL2}
Suppose that \(H=(T^2.(\Out(T)\times S_{2}))\cap A_n\) is a maximal subgroup of diagonal type in \(A_{n}\) such that \(n=|T|\) and \(T\) is any nonabelian simple group  except for those in {\rm(a)}--{\rm(d)}:
\begin{enumerate}[{\rm(a)}]
    \item \(T=\PSL_{3}(q)\) with prime power \(q\equiv-1\pmod{4}\);
    \item \(T=\PSU_{3}(q)\) with prime power \(q\equiv1\pmod{4}\);
    \item \(T=A_{7}\);
    \item \(T=M_{11}\).
\end{enumerate}
Let \(\pi\) be the projection of \(T^2.(\Out(T)\times S_{2})\) modulo \(T^2.\Out(T)\). Then \(\pi(H)=S_{2}\).
\end{lem}

\begin{proof}
Note that \(T^2.(\Out(T)\times S_{2})\) is a permutation group on the set $[T^2:D]$ of right cosets of $D$ in $T^2$, where \(D=\{(t,t)\mid t\in T\}\). Take $x=\sigma\tau\in T^2.(\Out(T)\times S_{2})$ with $\sigma=(1,1)\in T^2$ and \(\tau=(1,2)\in S_2\). Then for each \(D(t,1)\in [T^2:D]\),
\[
(D(t,1))^x=D(1,t)=D(t^{-1},1)\ \text{ and }\ (D(t,1))^{x^2}=D(t,1),
\]
and we note that \(x\) fixes \(D(t,1)\) if and only if \(t^{2}=1\).
Let $I$ be the number of involutions in $T$. Then \(x\) has \((|T|-I-1)/2\) cycles of length \(2\) and the number has the same parity as \((I+1)/2\).

If \(I\equiv 3\pmod{4}\), then \((I+1)/2\) is even and so $x\in A_n$, which implies that $x\in H$ and hence $\pi(H)=S_2$.
For the remainder of the proof, assume that \(I\not\equiv 3\pmod{4}\).
Then Lemma~\ref{marcel} asserts that \(T=\PSL_{2}(q)\), where \(q=p^f\equiv \varepsilon \pmod{8}\) with prime $p$ and $\varepsilon=\pm1$. Let $P=D.(\Out(T)\times S_2)=\Aut(T)\times S_{2}$.
Now $g\in\Aut(T)$ fixes a point $D(t,1)$ if and only if $D(t,1)^{g}=D(t^{g},1)=D(t,1)$, that is, if and only if $t\in\C_{T}(g)$.

Let $R\cong D_{2(q+\varepsilon)}$ be a dihedral subgroup of $\PGL_2(q)$ of order $2(q+\varepsilon)$. Then $R\cap\PSL_2(q)\cong D_{q+\varepsilon}$. Let $r$ be a generator of the cyclic subgroup of $R$ of order $q+\varepsilon$. Since $q\equiv\varepsilon\pmod{8}$, we have $\langle r\rangle=\langle r^2\rangle\times\langle r^{(q+\varepsilon)/2}\rangle$, and so $r^{(q+\varepsilon)/2}\notin\PSL_2(q)$. Clearly, $\C_{\PSL_2(q)}(r^{(q+\varepsilon)/2})=R\cap\PSL_2(q)$ of order $q+\varepsilon$. Then the action of $r^{(q+\varepsilon)/2}\in \Aut(T)$ on $[T^{2}:D]$ has exactly $q+\varepsilon$ fixed points. Then since $|r^{(q+\varepsilon)/2}|=2$, the element
\(r^{(q+\varepsilon)/2}\) has \((|\PSL_{2}(q)|-(q+\varepsilon))/2\) cycles of length \(2\) and this number is odd. So \(r^{(q+\varepsilon)/2}\notin H\).
Hence $(T^2.\Out(T))\cap A_n\cong T^2.C_f$ has index $2$ in $T^2.\Out(T)$, and so
\[
(T^2.(\Out(T)\times S_{2}))\cap A_n\cong T^2.(C_f\times S_{2}),
\]
which implies that \(\pi(H)=S_2\), completing the proof.
\end{proof}

\section{Homogeneous factorisations of wreath product}\label{4}

A group factorisation \(G=HK\) is said to be \emph{homogeneous} if \(H\cong K\). The following is a useful observation (see for instance~\cite[Lemma 3.2]{GLX2019}) in the study of homogeneous factorisations.

\begin{lem}\label{abp}
For any homogeneous factorisation \(G=AB\) we have \(\Pi(G)=\Pi(A)=\Pi(B)\) and \(|A|_{p}^{2}\geqslant |G|_{p}\) for any prime \(p\).
\end{lem}

This section is devoted to homogeneous factorisations of groups $G$ with \(A_{m}\wr A_{k}\leqslant G\leqslant S_{m}\wr S_{k}\). Our starting point is the following result from \cite[Lemma 3.5]{GLX2019}.

\begin{lem}\label{wr}
Let \(R\wr S_{k}\) be a wreath product with base group \(M=R_{1}\times\cdots\times R_{k}\), where \(R_{1}\cong\cdots\cong R_{k}\cong R\), and \(T\wr S_{k}\leqslant G\leqslant R\wr S_{k}\) such that \(T\leqslant R\). Suppose that \(G=AB\) is a homogeneous factorisation of \(G\) such that \(A\) is transitive on \(\{R_{1},\dots, R_{k}\}\). Then with \(\varphi_{i}(A\cap M)\) being the projection of \(A\cap M\) to \(R_{i}\), we have \(\varphi_{1}(A\cap M)\cong\cdots\cong \varphi_{k}(A\cap M)\) and \(\Pi(T)\subseteq\Pi(\varphi_{1}(A\cap M))\).
\end{lem}

In view of Lemmas~\ref{ab} and~\ref{wr}, we make the following hypothesis for convenience.

\begin{hyp}\label{gp}
Suppose that \(G=AB\) is a homogeneous factorisation such that \(A_{m}\wr A_{k}\leqslant G\leqslant S_{m}\wr S_{k}\) with \(m\geqslant 8\) and \(k\geqslant 2\). Let \(M=G\cap S_{m}^{k}\) and $\pi$ be the projection from \(G\) to \(G/M\). For $i\in\{1,\dots,k\}$ let \(\varphi_{i}\) be the  projections from \(M\) to the \(i\)-th component of $S_m^k$. Since \(\pi(A)\pi(B)=A_k\) or \(S_{k}\), we may assume without loss of generality that \(\pi(A)\) is transitive by Lemma~\ref{ab}. We also have \(\varphi_{1}(A\cap M)\cong\varphi_{2}(A\cap M)\cong\cdots\cong\varphi_{k}(A\cap M)\) and \(\Pi(A_{m})=\Pi(\varphi_{1}(A\cap M))\) by Lemma~\ref{wr}. Let \(q\) be the largest prime not greater than \(m\).
\end{hyp}

We derive necessary conditions for homogeneous factorisations $G=AB$ under Hypothesis \ref{gp}  in the remainder of this section, according to different values of $k$.

\begin{lem}\label{M}
Suppose that Hypothesis $\ref{gp}$ holds and $k\geqslant3$. Then $A$ has a minimal normal subgroup isomorphic to $A_\ell^k$ for some \(\ell\) with \(q\leqslant \ell\leqslant m\).
\end{lem}

\begin{proof}
For $i\in\{1,\dots,k\}$, let \(X_i=\varphi_{i}(A\cap M)\) and \(Y_i=\varphi_{i}(B\cap M)\). By Hypothesis \ref{gp}, we have \(\Pi(X_{i})=\Pi(A_{m})\). Then it follows from \cite[Theorem 4]{transitive} that \(A_{\ell}\trianglelefteq X_{i}\leqslant S_{\ell}\times S_{m-\ell}\) for some \(\ell\) with \(q\leqslant\ell\leqslant m\). Note that $q\geqslant7$ as $m\geqslant8$. Thus $\ell\geqslant7$, and so \(A_{\ell}\) is an insoluble composition factor of \(A\cap M\). Let \(f=\m_{A\cap M}(A_{\ell})\). We next prove $f=k$.

As Lemma~\ref{lem:bert} implies \(\ell\geqslant q>m/2\), we deduce from \(A_{\ell}\trianglelefteq X_{i}\leqslant S_{\ell}\times S_{m-\ell}\) that $X_i$ has a unique normal subgroup isomorphic to $A_\ell$ for all $i\in\{1,\dots,k\}$. Hence $A_\ell^k\cap A$ is a subdirect subgroup of $A_\ell^k$. Since \(\pi(A)\) is transitive, we obtain by Scott's Lemma (see \cite[Theorem 4.16]{csaba} and \cite[Corollary 4.17]{csaba}) that there is a block system of $\pi(A)$ on $\{1,\dots,k\}$ with $f$ blocks. Consequently, \(f\) divides \(k\), and \(A\) has a minimal normal subgroup \(W\cong A_{\ell}^{f}\).
Since \((m!)_{q}=(\ell!)_{q}=q\), any composition factor of \(A\cap M\) and \(X_{1}\) with order divisible by \(q\) is isomorphic to \(\A_{\ell}\), and \(W\) is the unique minimal normal subgroup of \(A\cap M\) with order divisible by \(q\).

If \(f\leqslant k/3\), then as $q\geqslant7$, we obtain, using Lemma \ref{p},
\[
|A|^{2}_{q}=|A\cap M|^{2}_{q}\cdot|\pi(A)|^{2}_{q}\leqslant|\varphi_{i}(A\cap M)|^{2f}_{q}\cdot(k!)^{2}_{q}<q^{2k/3}\cdot q^{2k/(q-1)}\leqslant q^{k}\leqslant |G|_q,
\]
contradicting Lemma \ref{abp}. Thus \(f=k/2\) or \(k\), as $f$ divides $k$. If $f=k$, then $W$ is a minimal normal subgroup of $A$ isomorphic to $A_\ell^k$ and the lemma holds. In particular, if $\pi(A)$ is primitive, then $f=k$, proving the lemma. Similarly, if $\pi(B)$ is primitive, then $B$ has a minimal normal subgroup isomorphic to $A_\ell^k$, and the lemma holds as $A\cong B$. To complete the proof, assume that neither $\pi(A)$ nor $\pi(B)$ is primitive and suppose for a contradiction that $f=k/2$. As a consequence, \(\pi(A)\) is an imprimitive subgroup of \(S_{k}\) with \(k/2\) blocks and therefore \(\pi(A)\leqslant S_{2}\wr S_{k/2}\) with $k\geqslant4$.

Since \(\pi(A)\pi(B)=\pi(G)\geqslant A_{k}\) and \(\pi(B)\) is not primitive, it follows from Theorem \ref{LPS} for $k\geqslant6$ and is easy to see for $k=4$ that \(A_{k-1}\leqslant \pi(B)\leqslant S_{k-1}\). Thus \(B\) has an orbit of length \(k-1\), say, \(\{Y_{1},\ldots, Y_{k-1}\}\), on \(\{Y_{1},\ldots,Y_{k}\}\). As a consequence, \(Y_{1}\cong\cdots\cong Y_{k-1}\).
Let \(p\geqslant5\) be a prime not exceeding \(m\). Then  \((m!)_{p}\geqslant p\) and
\begin{equation}\label{1}
        |B|_{p}=|B\cap M|_{p}\cdot|\pi(B)|_{p} \leqslant|Y_{1}|_{p}^{k-1}\cdot|Y_{k}|_{p}\cdot|S_{k}|_{p} \leqslant|Y_{1}|_{p}^{k-1}(m!)_{p}(k!)_{p}.
\end{equation}
By Lemma \ref{abp}, \(|B|_{p}^{2}\geqslant|G|_{p}\). This together with (\ref{1}) gives
\[
        (m!)_{p}^{k}(k!)_{p}=|G|_{p}\leqslant|B|_{p}^{2}\leqslant|Y_{1}|_{p}^{2(k-1)}(m!)_{p}^{2}(k!)_{p}^{2},
\]
        and so, using Lemma \ref{p},
\[
        |Y_{1}|_{p}^{2(k-1)}\geqslant(m!)^{k-2}_{p}/(k!)_{p}\geqslant p^{k-2}/(k!)_{p}>p^{k-2}p^{-k/(p-1)}\geqslant1
\]
as \(p\geqslant 5\) and \(k\geqslant3\). Hence \(|Y_{1}|_{p}>1\), which means that $|Y_1|$ is divisible by $p$.
Since \(m\geqslant 8\), the two largest primes not exceeding \(m\), say \(q\) and \(r\), are at least \(5\). Then \(|Y_{1}|\) is divisible by both \(q\) and \(r\). This together with \cite[Theorem 4 (ii)]{transitive} implies that \(A_{c}\trianglelefteq Y_{1}\leqslant S_{c}\times S_{m-c}\) for some \(q\leqslant c\leqslant m\). Let \(P=\prod_{i=1}^{k-1}\varphi_{i}(M)\) and \(d=\m_{B\cap P}(A_{c})\). Since \(\pi(B)\geqslant A_{k-1}\) acts primitively on \(\{ Y_{1},\ldots, Y_{k-1}\}\), we deduce by applying Scott's Lemma on the subgroup \(B\cap A_{c}^{k-1}\) of $A_c^{k-1}$ that either \(d=1\) or \(d=k-1\).

Suppose that \(d=1\). Then \(|B\cap P|_{q}=|A_{c}|_{q}=q\), and so
\[
|B|_{q}\leqslant |B\cap P|_{q}|Y_{k}|_{q}|\pi(B)|_{q}\leqslant q^{2}(k!)_{q}<q^2q^{k/(q-1)}\leqslant q^{(12+k)/6}.
\]
Moreover, Lemma \ref{abp} requires
\begin{equation}\nonumber
|B|_{q}^{2}\geqslant |G|_{q}=|M|_{q}(k!)_{q}=q^{k}(k!)_{q}\geqslant q^{k}.
\end{equation}
Hence $q^{(12+k)/3}>q^k$, and so $k=4$ as \(k\geqslant3\) is even.
Note that \(B\cap P\) has a unique normal subgroup \(K\) isomorphic to \(A_{c}\), see \cite[Corollary 4.17]{csaba}. Since \(B\cap P\) is normal in \(B\) and \(K\) is simple, we deduce that \(K\) is also minimal normal in \(B\). Since \(A\cong B\), it follows that \(A\) has a minimal normal subgroup \(H\cong A_{c}\). As \(\pi(A)\leqslant S_k=S_{4}\), we have \(H\leqslant A\cap M\). Notice that \(|H|_{q}=|A_{c}|_{q}>1\). This contradicts the fact that \(W\cong A_{\ell}^{f}=A_{\ell}^{k/2}\) is the unique minimal normal subgroup of \(A\cap M\) with order divisible by \(q\).

Thus we conclude that $d=k-1$.
Since \(\pi(B)\) is transitive on \(\{Y_{1},\ldots, Y_{k-1}\}\), Scott's Lemma implies that \(B\cap P\) has a unique minimal normal subgroup \(T\) isomorphic to \(A_{c}^{k-1}\). Moreover, \(B\cap P\) is normal in \(B\) as \(B\) fixes \(Y_{k}\). Thus \(T\) is a minimal normal subgroup of \(B\). Since \(A\cong B\), there exists a minimal normal subgroup \(L\) of \(A\) such that \(L\cong T\cong A_c^{k-1}\).
If \(L\cap(A\cap M)=1\), then \(L\lesssim\pi(A)\leqslant S_{k}\), which leads to
\[
q^{k-1}=|L|_{q}\leqslant|S_{k}|_{q}=(k!)_{q}<q^{k/(q-1)}\leqslant q^{k/6},
\]
a contradiction. Thus \(L\leqslant A\cap M\), and it follows that
\[
\{A_{c}\}=\CF(L)\subseteq\CF(A\cap M)=\CF(X_{1}).
\]
Recall that \(A_{\ell}\) is the only composition factor of \(X_{1}\) with order divisible by \(q\). We then conclude that \(c=\ell\). However, \(\m_{A\cap M}(A_{\ell})=k/2\) but \(\m_{A\cap M}(A_{c})\geqslant \m_{L}(A_{c})=k-1\), a contradiction. This completes the proof.
\end{proof}

\begin{lem}\label{k=2}
Suppose that Hypothesis $\ref{gp}$ holds and $k=2$. Then \(A_{\ell}\trianglelefteq \varphi_{i}(A\cap M)\leqslant S_{\ell}\times S_{m-\ell}\) for some \(\ell\) with \(q\leqslant\ell\leqslant m\). Moreover, if \(\ell<m\), then \(\m_{A\cap M}(A_{\ell})=2\) and $A$ has a minimal normal subgroup isomorphic to $A_\ell^2$.
\end{lem}

\begin{proof}
For $i\in\{1,2\}$, let \(X_i=\varphi_{i}(A\cap M)\) and \(Y_i=\varphi_{i}(B\cap M)\). By Hypothesis \ref{gp}, we have \(\Pi(X_{i})=\Pi(A_{m})\). Then it follows from \cite[Theorem 4]{transitive} that \(A_{\ell}\trianglelefteq X_{i}\leqslant S_{\ell}\times S_{m-\ell}\) for some \(\ell\) with \(q\leqslant\ell\leqslant m\). Note that $q\geqslant7$ as $m\geqslant8$. Thus $\ell\geqslant7$, and so \(A_{\ell}\) is an insoluble composition factor of \(A\cap M\). Let \(f=\m_{A\cap M}(A_{\ell})\). We next assume \(\ell<m\) and prove that $f>1$, which will then imply the conclusion in the ``Moreover'' part of the lemma. Suppose for a contradiction that $f=1$.
If \(\ell=m-1\), then \(|A|\leqslant 2^{2}|A_{m-1}|\) and so $|A|^{2}\leqslant2^2((m-1)!)^{2}<(m!)^{2}\leqslant|G|$, contradicting Lemma~\ref{abp} since \(G=AB\) is a homogeneous factorisation. Thus we have \(\ell\leqslant m-2\).

We first prove \(\pi(B)=C_{2}\).
Suppose for a contradiction that \(\pi(B)=1\). Then \(B\leqslant M\) and therefore we derive from $G=AB$ that \(M=(A\cap M)B\). Hence
\(\varphi_{i}(M)=\varphi_{i}(A\cap M)\varphi_{i}(B)\) for \(i\in\{1,2\}\). Note that \(A_{\ell}\leqslant \varphi_{i}(A\cap M)\leqslant S_{\ell}\times S_{m-\ell}\). By checking the factorisations of alternating and symmetric groups given in Theorem~\ref{LPS}, we conclude that either \(A_{m}\leqslant \varphi_{i}(B)\), or \(\varphi_{i}(B)\) is \((m-\ell)\)-homogeneous with \(1\leqslant m-\ell\leqslant 5\). If \(A_{m}\leqslant \varphi_{i}(B)\) for some $i\in\{1,2\}$, then
\[
A_{m}\in \ICF(B)=\ICF(A)=\ICF(A\cap M)\subseteq\ICF(X_1)\cup\ICF(X_2),
\]
contradicting the condition that \(A_{\ell}\leqslant X_j\leqslant S_{\ell}\times S_{m-\ell}\) for $j\in\{1,2\}$.
Therefore, \(Y_i=\varphi_{i}(B)\) is \((m-\ell)\)-homogeneous for \(i\in\{1,2\}\). Since
\[
A_{\ell}\in \ICF(A)=\ICF(B)=\ICF(Y_1)\cup\ICF(Y_2),
\]
we deduce that \(A_{\ell}\in\ICF(Y_1)\) or \(\ICF(Y_2)\). However, by Lemma \ref{LPS}, this is impossible as both \(Y_1\) and \(Y_2\) are \((m-\ell)\)-homogeneous. Thus \(\pi(B)=C_{2}\), as claimed.

From \(\pi(B)=\C_{2}\) we conclude that \(Y_{1}\cong Y_{2}\). Then Lemma~\ref{wr} implies that \(\Pi(Y_1)=\Pi(A_{m})\), and it follows from~\cite[Theorem~4]{transitive} that \(A_{t}\trianglelefteq Y_1\leqslant S_{t}\times S_{m-t}\) for some \(t\) with \(q\leqslant t\leqslant m\). As a consequence, \(A_{t}\in\ICF(B)=\ICF(A)=\ICF(A\cap M)\). Since \(A_{\ell}\) is the unique insoluble composition factor of \(A\cap M\) with order divisible by \(q\), we conclude that \(t=\ell\). Now \(A_\ell\trianglelefteq X_{i}\leqslant S_\ell\times S_{m-\ell}\) and \(A_\ell\trianglelefteq Y_{i}\leqslant S_\ell\times S_{m-\ell}\) for \(i\in\{1,2\}\). By Lemma~\ref{conju}, we have \(H'\trianglelefteq X_{i}\leqslant H\times K\) and \((H^g)'\trianglelefteq Y_{i}\leqslant H^g\times K^g\) for \(i\in\{1,2\}\), where $H=S_\ell$ is the pointwise stabiliser of an $(m-\ell)$-subset in $S_m$, $K=S_{m-\ell}$ is the centraliser of $H$ in $S_m$, and $g$ is some element in $A_m$. Let $P=H\times H<S_m\times S_m$.
Since $\pi(A)=C_2$ while \(\m_{A\cap M}(A_{\ell})=1\), there exists \(\sigma\in\Aut(H')\) such that \(A\cap P'=\{(x, x^{\sigma}) \mid x\in H'\}\).
Moreover, since $\Aut(H')\cong S_\ell\cong H$ (as \(\ell>6\)), we may view $x^\sigma$ as the conjugate of \(x\) by $\sigma\in H$.
Similarly, there exists \(\tau\in H^g\) such that \(B\cap P'=\{(y, y^{\tau}) \mid y\in (H^g)'\}\).

Since \(\ell\leqslant m-2\), there exists some odd permutation $\gamma\in K=S_{m-\ell}$. If $\sigma^{-1}g\tau\in A_m$, then take $\delta=(g,\sigma^{-1}g\tau)\in A_m\times A_m$. If $\sigma^{-1}g\tau\in S_m\setminus A_m$, then take $\delta=(g,\gamma\sigma^{-1}g\tau)\in A_m\times A_m$. It follows that
\[
(A\cap P')^\delta=\{(x,x^\sigma)\mid x\in H'\}^\delta=\{(x^g,(x^g)^\tau)\mid x\in H'\}=\{(y, y^\tau)\mid y\in (H^g)'\}=B\cap P'.
\]
%This means that \(A\cap P'\) and \(B\cap P'\) are conjugate in \(M\).
Since \(A\cap P'\) and \(B\cap P'\) are normal in \(A\) and \(B\), respectively, this implies that
\[
G=AB=\N_{G}(A\cap P')\N_{G}(B\cap P')=\N_{G}(A\cap P')\N_{G}((A\cap P')^\delta)=\N_{G}(A\cap P')(\N_{G}(A\cap P'))^\delta,
\]
which contradicts Lemma~\ref{hab1} as $A\cap P'\cong A_\ell$ is not normal in $G$.
\end{proof}

\begin{lem}\label{diagonal}
Suppose that Hypothesis $\ref{gp}$ holds and $k=2$. If \(\Soc(A\cap M)=\{(x,x^\sigma)\mid x\in A_{m}\}\) for some \(\sigma\in \Aut(A_{m})\), then \(\Soc(B\cap M)\neq\{(y,y^\gamma)\mid y\in A_{m}\}\) for any \(\gamma \in \Aut(A_{m})\).
\end{lem}

\begin{proof}
By Hypothesis \ref{gp}, \(m\geqslant8\). The conclusion for $m\in\{8,9,10\}$ can be directly verified by computation in \textsc{Magma}~\cite{BCP1997}. For the rest of the proof we assume \(m\geqslant11\).
Suppose for a contradiction that \(\Soc(B\cap M)=\{(y,y^\gamma)\mid y\in A_{m}\}\) for some \(\gamma \in \Aut(A_{m})\).
Note that \(|A|=|\Soc(A\cap M)|\delta=|A_{m}|\delta\) where \(\delta=2^{i}\) with \(1\leqslant i\leqslant 3\). It follows that
\[
|A_{m}|^{2}\delta\leqslant |G|=\frac{|A||B|}{|A\cap B|}=\frac{|A_{m}|^{2}\delta^{2}}{|A\cap B|},
\]
and so \(|A\cap B|\leqslant\delta\leqslant 8\). Then since
\[
A\cap B\cap M'=(A\cap M')\cap(B\cap M')=\{(x,x^\sigma)\mid x\in A_{m},\,x^\sigma=x^\gamma\}
\]
has the same cardinality as $\{x\in A_{m}\mid x^{\sigma\gamma^{-1}}=x\}=\C_{A_{m}}(\sigma\gamma^{-1})$, we deduce that \(|\C_{A_{m}}(\sigma\gamma^{-1})|\leqslant|A\cap B|\leqslant8\).
In the following we show however that, \(|\C_{A_{m}}(x)|>8\) for all \(x\in A_{m}\) when \(m\geqslant11\). This will complete the proof.

For \(11\leqslant m\leqslant 16\) the conclusion \(|\C_{A_{m}}(x)|>8\) can be verified by  \textsc{Magma}~\cite{BCP1997} directly. For \(m\geqslant17\), we prove the conclusion by induction. Suppose that the result is true for \(m-1\). Then for \(x\in A_{m}\), write \(x=\alpha_{1}\cdots\alpha_{s}\), where \(\alpha_{i}\) are disjoint cycles of size \(r_{i}\). If there exists some \(i\) with \(r_{i}=1\), then \(x\in A_{m-1}\) and by induction, \(|\C_{A_{m}}(x)|\geqslant|\C_{A_{m-1}}(x)|>8\). Now assume \(r_{i}\geqslant 2\) for all \(i\). Then
\[
2|\C_{A_{m}}(x)|\geqslant|\C_{S_{m}}(x)|\geqslant|\alpha_{1}|\cdots|\alpha_{s}|=r_{1}\cdots r_{s}\geqslant \sum_{i=1}^{s}r_{i}=m\geqslant 17,
\]
and so \(|\C_{A_{m}}(x)|\geqslant m/2>8\), as required.
\end{proof}

\section{Proof of the main theorem}\label{3}

The following lemma provides necessary information to address the excluded candidates for $T$ in Cases~(a)--(d) of Lemma~\ref{PSL2}.

\begin{lem}\label{T}
Let \(G=(T_{1}\times T_{2}).\Out(T)\), where \(T_{1}\cong T_{2}\cong T\) and \(T\) is one of \(A_{7}\), \(M_{11}\), \(\PSL_{3}(q)\) with \(q\equiv -1\pmod{4}\) or \(\PSU_{3}(q)\) with \(q\equiv 1\pmod{4}\). Suppose that \(H\leqslant G\) with \(|H|_{r}\geqslant |G|_{r}^{2/3}\) for all \(r\in\Pi(G)\). Then \(T_{1}\times T_{2}\leqslant H\).
\end{lem}

\begin{proof}
For \(T=A_{7}\), \(M_{11}\), \(\PSL_{3}(3)\) or \(\PSU_{3}(5)\), calculation in \textsc{Magma}~\cite{BCP1997} directly verifies the result.
In what follows, we assume that \(T\) is \(\PSL_{3}(q)\) with \(3<q \equiv -1\pmod{4}\) or \(\PSU_{3}(q)\) with \(5<q \equiv 1\pmod{4}\). Let $q=p^f$ with prime $p$, and let $e=1$ or $2$ according to $T=\PSL_3(q)$ or $\PSU_3(q)$, respectively. Then \(p^{3ef}-1\) has a primitive prime divisor $s$, such that $s$ divides $|T|$ but not \(|\Out(T)|\). If $p=3$, then \(p^{2f}-1\) has a primitive prime divisor \(t\), such that $|\Out(T)|_t=1<|T|_t$. If $p\geqslant 5$, then by letting $t=p$, we have \(|\Out(T)|_t=f_t\leqslant f<p^f<|T|_{t}\). In either case, we have \(|\Out(T)|_{r}<|T|_{r}\) for each $r\in\{s,t\}$. It follows that, for $i\in\{1,2\}$,
\[
|H\cap T_{i}|_{r}\cdot|T|_{r}\cdot|\Out(T)|_{r}\geqslant|H|_{r}\geqslant|G|_{r}^{2/3}=|T|_r^{4/3}|\cdot|\Out(T)|_r^{2/3},
\]
and so $|H\cap T_{i}|_{r}\geqslant|T|_r^{1/3}|\Out(T)|_r^{-1/3}>1$. Hence we conclude from \cite[Table~10.3]{transitive} that \(H\cap T_{i}=T_{i}\) for $i\in\{1,2\}$, which means \(T_{1}\times T_{2}\leqslant H\), as required.
\end{proof}

The next lemma proves part~(a) of Theorem~\ref{thmm} under the assumption that $G_v$ is a maximal subgroup of $G$ of type~(f) in Theorem~\ref{onanscottthm}.

\begin{lem}\label{AS}
Let $\Gamma$ be a $G$-vertex-primitive $(G,s)$-arc-transitive digraph such that $s\geqslant2$ and $G$ is almost simple with socle $A_n$, and let \(v\) be a vertex of $\Gamma$. Suppose that $T\trianglelefteq G_v\leqslant\Aut(T)$ for some nonabelian simple group $T$. Then $T=\Sp_4(2^f)$ or $\mathrm{P\Omega}_8^+(q)$, where $f\geqslant2$ and $q$ is a prime power.
\end{lem}

\begin{proof}
Let $(u,v,w)$ be a $2$-arc of $\Gamma$. According to~\cite[Lemma~2.2]{GX2018}, there exists $g\in G$ such that $G=\langle G_v,g\rangle$ and $G_v=G_{uv}G_{vw}$ with $G_{uv}^g=G_{vw}$. Thus, by~\cite[Proposition~3.3]{GLX2019}, one of the following appears:
\begin{enumerate}[{\rm(i)}]
    \item both $G_{uv}$ and $G_{vw}$ contain $T$;
    \item \(T=A_6\) and $\Soc(G_{uv})=A_5$, or $T=M_{12}$ and $\Soc(G_{uv})=M_{11}$;
    \item $T=\Sp_4(2^f)$ or $\mathrm{P\Omega}_8^+(q)$, where $f\geqslant2$ and $q$ is a prime power.
\end{enumerate}
If~(i) appears, then $T^g=\Soc(G_{uv})^g=\Soc(G_{uv}^g)=\Soc(G_{vw})=T$, and so $T\trianglelefteq\langle G_v,g\rangle=G$, a contradiction.  Next suppose that (ii) holds.   Since $G_v$ is maximal in $G$, it is a primitive permutation group of degree $n$ and so $n$ is the index of a core-free maximal subgroup of $G_v$. When $G_v=A_6$ or $S_6$  this implies that $n=10$ or 15, but a  \textsc{Magma}~\cite{BCP1997}  calculation shows that in all cases $G_v$ is not maximal in $A_n$ or $S_n$. Thus $(G,G_v,G_{uv})=(A_n,M_{12},M_{11})$.  However, a computation in  \textsc{Magma}~\cite{BCP1997}  shows that for all the values of $n$ where two nonconjugate subgroups of $G_v$ isomorphic to $M_{11}$ are conjugate in $A_n$, we have that $\mathrm{N}_{A_n}(M_{12})=\Aut(M_{12})$, contradicting $G_{v}$ being maximal in $G$.
Hence~(iii) holds, as the lemma states.
\end{proof}

We are now ready to prove Theorem~\ref{thmm}.

\begin{proof}[Proof of Theorem~$\ref{thmm}$]
Suppose that the assumptions of Theorem~$\ref{thmm}$ hold, and in particular that $s\geq2$.  Let $(u,v,w)$ be a $2$-arc of $\Gamma$. Then there exists $g\in G$ such that $(u^g,v^g)=(v,w)$, and there is a homogeneous factorisation $G_v=G_{uv}G_{vw}$ by~\cite[Lemma 2.2]{GX2018}. We first fix the gap in the proof of Theorem~\ref{PWY}. As mentioned in Remark~\ref{rem}, the gap occurs when $G=A_n$ and $G_v$ is a maximal subgroup of $G$ in Cases~(b),~(c) and~(e) of Theorem~\ref{onanscottthm}. By Lemmas~\ref{imprimitive},~\ref{PA} and~\ref{PSL2}, this possibly happens only if $G_v=T^2.\Out(T)$ is the holomorph of some simple group $T$ in Cases~(a)--(d) of Lemma~\ref{PSL2}. We prove for these cases that $s\leqslant2$ still holds, as Theorem~\ref{PWY} states. Suppose for a contradiction that $s\geqslant3$. Then since the valency of $\Gamma$ is $|G_v|/|G_{uv}|=|G_v|/|G_{vw}|$, it follows that $|G_v|$ is divisible by $(|G_v|/|G_{uv}|)^3=(|G_v|/|G_{vw}|)^3$. Hence \(|G_{uv}|_{r}=|G_{vw}|_{r}\geqslant |G_v|_{r}^{2/3}\) for all \(r\in\Pi(G_v)\), and then by Lemma~\ref{T}, both $\Soc(G_{uv})$ and $\Soc(G_{vw})$ contain $T^2$. As a consequence, $\Soc(G_{uv})=\Soc(G_v)=\Soc(G_{vw})$, and so
\[
\Soc(G_v)^g=\Soc(G_{uv})^g=\Soc(G_{uv}^g)=\Soc(G_{vw})=\Soc(G_v),
\]
contradicting~\cite[Lemma~2.14]{GLX2019}. This fixes the gap in the proof of Theorem~\ref{PWY}.

We will next prove that \(s=2\). By Theorem~\ref{PWY}, we need to show that for the groups in Theorem~\ref{PWY}(2) the value of $s$ is $2$. So assume that, for some integers $m\geq 8$ and $k\geq2$ such that $n=mk, m^k$ or $(m!/2)^{k-1}$, we have
\[
(A_m\wr S_k)\cap G\leqslant G_v\leqslant (S_m\wr S_k)\cap G.
\]
In particular $\Soc(G_v)=A_m^k$.

\medskip\noindent
{\it Claim:\ For these groups, $s=2$ and also \(k=2\) and \(\m_{G_{uv}\cap M}(A_{m})=1\).}

Let \(q\) be the largest prime not exceeding $m$.
By Lemmas~\ref{M} and~\ref{k=2}, there is a composition factor $A_\ell$ of $G_{uv}\cap\Soc(G_v)$ with $q\leqslant\ell\leqslant m$ such that, if \((k,\ell)\neq (2,m)\), then $G_{uv}$ has a minimal normal subgroup isomorphic to $A_\ell^k$. Let \(M=\prod_{i=1}^{k}M_{i}\cong S_{m}^{k}\) be the base group of $S_m\wr S_k$, where $M_i=S_m$ for $i\in\{1,\dots,k\}$, and let \(N_{i}=\Soc(M_{i})\cong A_{m}\) for \(i\in\{1,\dots,k\}\), and \(N=\prod_{i=1}^{k}N_{i}\).

Suppose that $R_{1}$ is a minimal normal subgroup of \(G_{uv}\), isomorphic to $A_\ell^k$. It follows that \(R_{2}:=R_{1}^{g}\) is a minimal normal subgroup of \(G_{vw}\). If \(R_{i}\cap M=1\) for some $i\in\{1,2\}$, then \(\A_{\ell}^{k}\cong R_i\lesssim S_{k}\), and so
\[
q^{k}\leqslant|\A_{\ell}|^{k}_{q}\leqslant |S_{k}|_{q}=(k!)_{q}<q^{\frac{k}{q-1}},
\]
which is impossible. Therefore, \(R_{i}\leqslant M\) for each $i\in\{1,2\}$. If \(\ell=m\), then since \(R_{1}=\Soc(M)\) is the unique subgroup of \(M\) isomorphic to \(A_{m}^{k}\), we deduce that \(R_{1}^{g}=R_{2}=R_{1}\), contradicting~\cite[Lemma~2.14]{GLX2019}. Thus we have \(\ell<m\).
For $i\in\{1,2\}$, we deduce from $N_j\unlhd M$ and $R_i\leqslant M$ that $R_i\cap N_j\unlhd R_i$.
Since $R_i\cong A_\ell^k$, it follows that $R_i\cap N_j=A_\ell^t$ for some $t\in\{0,1,\dots,k\}$.
Accordingly, $q^t\leqslant|A_\ell^t|_q=|R_i\cap N_j|_q\leqslant|N_j|_q=|A_m|_q=q$, that is, $t\leqslant1$.
If $t=0$, then $R_i\cong R_iN_j/N_j\leqslant N/N_j\cong A_m^{k-1}$, and so $q^k\leqslant|A_\ell^k|_q=|R_i|_q\leqslant|A_m^{k-1}|_q=q^{k-1}$, a contradiction.
Thus $R_i\cap N_j=A_\ell$ for each $j\in\{1,\dots,k\}$, which implies that
\[
R_i=(R_i\cap N_1)\times(R_i\cap N_2)\times\cdots\times(R_i\cap N_k).
\]
For each $j\in\{1,\dots,k\}$, by Lemma~\ref{conju}, there exists $g_j\in N_j$ such that $(R_1\cap N_j)^{g_j}=R_2\cap N_j$.
Let $g=g_1\cdots g_k$. Then $g\in N\leqslant G_v$ such that $R_1^g=R_2$. Since $R_1\unlhd G_{uv}$, it follows that $R_2\unlhd G_{uv}^g$.
This together with $R_2\unlhd G_{vw}$ leads to $R_2\unlhd\langle G_{uv}^g,G_{vw}\rangle$.
However, by \cite[Lemma 2.11]{GLX2019} and Lemma~\ref{hab1} we have $G_v=G_{uv}G_{vw}=G_{uv}^gG_{vw}$.
Hence $A_\ell^k\cong R_2\unlhd\langle G_{uv}^g,G_{vw}\rangle=G_v$, contradicting $\Soc(G_v)=A_m^k$.

Thus we conclude that $G_{uv}$ has no minimal normal subgroup isomorphic to $A_\ell^k$. Consequently, \((k,\ell)=(2,m)\) and \(\m_{G_{uv}\cap M}(A_{m})=1\). This implies that \(|G_{uv}|_q=|\A_{m}|_q=q\). Since \(|G_{v}|_q=|A_{m}|_q^{2}=q^{2}\), we deduce that the valency of $\Gamma$ has $q$-part $|G_{v}|_q/|G_{uv}|_q=q^2/q=q$.
If \(s\geqslant3\), then \(|G_{v}|_{q}\geqslant q^{3}\), a contradiction. Therefore, \(s=2\) and the Claim is proved.

It follows from \cite[Theorem~1.1~and~Lemmas~3.2--3.3]{PWY2020} that either \(T\trianglelefteq G_{v}\leqslant \Aut(T)\) for some nonabelian simple group \(T\), or \(G_{v}=(T^{k}.(\Out(T)\times S_{k}))\cap G\) with \(n=|T|^{k-1}\) for some almost simple group \(T\) and integer \(k\geqslant2\),  or part (2) of Theorem~$\ref{PWY}$ holds.
Further, by Lemma~\ref{AS}, if $G_v$ is almost simple then part (a) of Theorem~\ref{thmm} holds. Thus to complete the proof of Theorem \ref{thmm}, it remains to show that, if  Theorem~$\ref{PWY}$ (2) holds, then $k=2$ and  \(n=(m!/2)^{k-1}\) so  that  \(G_{v}\) is a maximal subgroup of diagonal type of \(G\).

Therefore we may assume that \((A_{m}\wr S_{k})\cap G\leqslant G_{v}\leqslant(S_{m}\wr S_{k})\cap G\)
 with \(m\geqslant8, k>1\) and \(n=mk, m^{k}\) or \((m!/2)^{k-1}\).
 From the Claim above, for these groups we must have $k=2$, and  \(\m_{G_{uv}\cap M}(A_{m})=1\).

Let $\pi$ be the projection from \(G_v\) to \(G_v/M\), and let \(\varphi_{i}\) be the  projections from \(M\) to $M_i$ for $i\in\{1,2\}$.
Since \(\pi(G_{uv})\pi(G_{vw})=\pi(G_{v})=S_{2}\), this implies that at least one of \(\pi(G_{uv})\) or \(\pi(G_{vw})\) is \(S_{2}\). Without loss of generality, assume  \(\pi(G_{uv})=S_{2}\). Then \(\varphi_{1}(G_{uv}\cap M)\cong\varphi_{2}(G_{uv}\cap M)\). This together with the fact that \(\m_{G_{uv}\cap M}(A_{m})=1\) implies that \[G_{uv}^{(\infty)}=\Soc(G_{uv}\cap M)=\{(x,x^\sigma)\mid x\in A_{m}\}\] for some \(\sigma\in \Aut(A_{m})\). We then deduce from Lemma \ref{diagonal} that \(G_{vw}^{(\infty)}=N_{i}\) for some $i\in\{1,2\}$. Observe that \((G_{uv}^{(\infty)})^{g}=G_{vw}^{(\infty)}\) as \(G_{uv}^{g}=G_{vw}\).
If \(G_{v}\) is an imprimitive maximal subgroup of \(G\), then \(n=2m\) and \(G_{uv}^{(\infty)}\) has no fixed points on \(\{1,\ldots,2m\}\) but \(G_{vw}^{(\infty)}=N_i\) fixes $m$ points, and we have a contradiction.
Suppose next that \(G_{v}\) is a maximal subgroup of \(G\) in product action. Then \(n=m^{2}\) and \(G_{uv}^{(\infty)}\) is conjugate to \(D=\{(x,x)\mid x\in A_{m}\}\) in \(S_m\wr S_2<S_n\). Note that \(D\) has an orbit of length \(m(m-1)\), while the orbits of \(N_i\) are all of length \(m\).
Thus \(G_{uv}^{(\infty)}\) and \(G_{vw}^{(\infty)}\) are not conjugate in \(S_{n}\), and again we have a contradiction.
Consequently, \(G_{v}\) is a maximal subgroup of \(G\) of diagonal type, completing the proof.
\end{proof}

\end{document}